\documentclass{amsart}
\usepackage[english]{babel}
\usepackage[utf8]{inputenc}
\usepackage{hyperref}
\usepackage[foot]{amsaddr}

\usepackage{amssymb,amsmath,amsthm}
\usepackage{xfrac}
\usepackage{breqn}
\numberwithin{equation}{section}
\usepackage{graphicx}
\usepackage{tikz}
\usepackage{mathtools}
\usepackage{cleveref}
\usepackage{amsrefs}
\newtheorem{theorem}{Theorem}
\newtheorem{definition}{Definition}
\newtheorem{proposition}{Proposition}
\newtheorem{corollary}{Corollary}
\newtheorem{lemma}{Lemma}

\title{Automorphisms of the worm domain}

\author{Fani Xerakia}

\address{Faculty of Mathematics, University of Vienna\\
Oskar-Morgenstern-Platz 1\\
1090 Vienna, Austria}

\subjclass[2020]{32T20, 32V15 (primary), 32V20}
\email{fani.xerakia@univie.ac.at}

\begin{document}

\begin{abstract}
The Diederich--Forn\ae ss worm domain, an important example of a smoothly bounded pseudoconvex domain without a Stein neighborhood basis, provides key counterexamples in the theory of several complex variables. In this paper, we examine its automorphism group and observe that its boundary is locally spherical everywhere except at the exceptional locus and the closure of the caps.
\end{abstract}
\maketitle
\section{Introduction}
Pseudoconvex domains play a central role in the theory of several complex variables. In this paper, pseudoconvexity is understood in the sense of Levi pseudoconvexity, namely that the Levi form of a defining function is positive semi-definite on the boundary. The solution of the Levi problem shows that pseudoconvex domains are precisely the domains of holomorphy, that is, domains on which holomorphic functions cannot be extended across the boundary. In this sense, they generalize the notion of domains in one complex variable, where every domain is a domain of holomorphy.

Any pseudoconvex domain has an exhaustion by an increasing sequence of \linebreak[4] smoothly bounded, strongly pseudoconvex subdomains. This naturally leads to the dual question: can one approximate the closure of a pseudoconvex domain from the exterior by a decreasing sequence of smoothly bounded pseudoconvex domains? A domain with this property is said to admit a Stein neighborhood basis.

The Hartogs triangle
\begin{align*}
    \mathcal{D}=\left\{(z,w) \in \mathbb{C}^2:0<|z|<|w|<1 \right\}
\end{align*}
provides a classical example of a domain in $\mathbb{C}^2$ that lacks a Stein neighborhood basis. Initially, it was believed that the reason for this was the non-smoothness of the boundary.

To address this question, Diederich and Fornæss constructed the worm domain in 1977 \cite{Diederich1977} as the first example of a smoothly bounded pseudoconvex domain without a Stein neighborhood basis.

\begin{definition} \label{defworm}
    The Diederich--Forn\ae ss worm domain is the set
\begin{align*}
    \mathcal{W}_{\mu} = \left\{ (z,w) \in \mathbb{C}^2 : \left| w - e^{i \log |z|^2} \right|^2 
    < 1-\eta \left( \log |z|^2 \right), z \neq 0 \right\}
\end{align*}
where $\eta: \mathbb{R} \to \mathbb{R}$ is a smooth function satisfying:
\begin{itemize}
    \item[(i)] $\eta \geq 0$, even and convex;
    \item[(ii)] $\eta^{-1}(0) = [-\mu, \mu],$ for some $\mu >0;$
    \item[(iii)] there exists $a > 0$ such that $\eta(x) > 1$ if $|x| > a;$
    \item[(iv)] $\eta'(x) \neq 0$ whenever $\eta(x) = 1.$
\end{itemize}
\end{definition}

For notational convenience, we will omit the subscript $\mu$ and refer to the worm domain simply as $\mathcal{W}.$ 

The boundary of the worm domain is strongly pseudoconvex except along an annular set of weakly pseudoconvex points, which will be referred to as the exceptional locus throughout this paper. For further details, we refer the reader to Section 2.

Since its introduction, the worm domain has attracted considerable attention by several researchers due to its unusual geometric and analytic properties. It has played an important role in the study of pseudoconvexity, the $\bar{\partial}$-Neumann problem, and Bergman kernel theory; see, among many others, \cite{Barrett1992}, \cite{BarrettSahutoglu2012},  \cite{BoasStraube1990}, \cite{Christ1996GlobalI}, \cite{Kiselman1991}, \cite{krantzpeloso}, \cite{KrantzPeloso2008}. In addition, several unbounded and higher-dimensional versions of the worm domain have been extensively studied; see, for example, \cite{https://doi.org/10.1112/blms.70057}, \cite{KrantzMonguzziPelosoStoppato2023}, \cite{KrantzPelosoStoppato2016}.

The worm domain has also served as a fundamental counterexample in several complex variables, challenging many previously held assumptions.

Let $\Omega_1$ and $\Omega_2$ be smoothly bounded domains, and let $F: \Omega_1 \to \Omega_2$ be a biholomorphism. A natural question is whether $F$ extends as a $\mathrm{C}^{\infty}$ diffeomorphism $\Tilde{F}: \overline{\Omega}_1 \to \overline{\Omega}_2.$ If both domains are strongly pseudoconvex, Fefferman's result \cite{Fefferman1974-an}, shows that such an extension exists. For Levi pseudoconvex domains, Bell \cite{feef0656-082d-3749-99a7-2cad1870505b} and Bell-Ligocka \cite{Bell1980-ev} showed that the extension remains possible, provided one of the domains satisfies Condition $\mathrm{R},$ a condition formulated by Bell \cite{feef0656-082d-3749-99a7-2cad1870505b}; see Section 2 for further details.

However, the worm domain does not satisfy Condition $\mathrm{R}$ \cite{Christ1996GlobalI}. Therefore, one may consider a weaker version of the problem: whether a biholomorphic automorphism of the worm extends smoothly to the boundary. This question has an affirmative answer, since the automorphism group of the worm domain consists only of rotations in the $z$-variable, as stated by Chen in \cite{MR1116311}:
\begin{theorem}
The automorphism group of $\mathcal{W}$ consists only of maps of the form $$(z,w) \mapsto (e^{i \theta}z,w), \text{ } \theta \in \mathbb{R},$$ i.e., rotations in the $z$-variable.
\end{theorem}

While the result is widely accepted, several researchers have noted difficulties in reproducing certain parts of the original argument. Motivated by this, we provide an alternative proof of the theorem above in this paper, which may also be of independent geometric interest. In particular, this new approach yields the following additional result:
\begin{corollary} \label{locsph1}
    The boundary of $\mathcal{W}$ is locally spherical except at the exceptional locus and the closure of the caps.
\end{corollary}
Here, local sphericity means that, in a neighborhood of each boundary point outside the exceptional locus and the closure of the caps, the boundary of the worm domain is locally biholomorphically equivalent to the sphere in $\mathbb{C}^2$.

The above result, together with the explicit local biholomorphism between the worm and the sphere given in Section 3, determines, in principle, the structure of the automorphism group of the worm. The computation of these automorphisms, however, is of independent geometric interest, and the resulting formulas are not immediately obvious. For this reason, an explicit computation is included.

This article is organized as follows:

Section 2 collects key definitions and preliminary results that will be used in the subsequent sections.

Section 3 considers the unbounded worm domain. In particular, it provides a classification of the CR automorphisms of the smooth strongly pseudoconvex part of its boundary, as well as a proof of its local sphericity near all smooth strongly pseudoconvex points.

Section 4 contains the proof concerning the automorphisms of the Diederich--Fornæss worm domain.

\section{Definitions and results}
In this section, we introduce several key definitions and preliminary results that will be used in the analysis of the worm domain.

We begin by fixing some notation. Let $z=(z_1,\dots,z_N)\in\mathbb{C}^N$, where $z_j=x_j+iy_j$ with $x_j,y_j\in\mathbb{R}$, $j=1,\dots,N$. Using the standard identification between $\mathbb{C}^N$ and $\mathbb{R}^{2N}$, we denote a function defined on a subset $\Omega\subset\mathbb{C}^N$ by $f(x,y)$, where $x=(x_1,\dots,x_N)$ and $y=(y_1,\dots,y_N)$, or more conveniently, by $f(z,\overline{z})$.

If $f$ is real-analytic, we may complexify it and consider the function $f(z,\zeta)$, where $z,\zeta\in\mathbb{C}^N,$ and continue to denote the complexified function by $f.$ Throughout this paper, we will freely interchange between the complexified and non-complexified versions of real-analytic functions, using the same notation whenever no confusion can arise.

\subsection{Segre Varieties} 
We begin with the key notion of Segre varieties, which are important tools for the finite jet determination of mappings. More precisely:
\begin{definition}
    Let $M \subset \mathbb{C}^N$ be a real-analytic hypersurface with local real-analytic defining function $\rho$ defined in some open subset $V \subset \mathbb{C}^N$ and $p \in M.$ The (first iterated) Segre variety of $p$ is defined as
$$\mathcal{S}_p(U)=\mathcal{S}^1_p(U)= \left\{ z \in U: \rho(z, \bar{p})=0 \right\},$$ where $U \subset V$ is a small neighbourhood of $p.$ 
The iterated Segre varieties are defined inductively as    $$\mathcal{S}^k_p(U)= \bigcup_{q \in \mathcal{S}_p^{k-1}(U) } \mathcal{S}_q(U).$$
\end{definition}

It is typical to omit the neighborhood $U$ from the notation and assume that the statements hold in a neighborhood around the point $p.$

\begin{proposition}
Let $M \subset \mathbb{C}^N$ and $M' \subset \mathbb{C}^{N'}$ be real-analytic hypersurfaces with defining functions $\rho$ and $\rho'$, respectively. Let $\mathcal{S}_p$ and $\mathcal{S}'_{p'}$ denote the corresponding Segre varieties at points $p \in M$ and $p' \in M'$. Then the following properties hold:
\begin{itemize}
\item[(i)] $p \in \mathcal{S}_p$
\item[(ii)] For $q \in M$ near $p$, one has $q \in \mathcal{S}_p$ if and only if $p \in \mathcal{S}_q$
\item[(iii)] Let $U$ be a neighbourhood of $p$ in $\mathbb{C}^N$ and $H: U \to \mathbb{C}^{N'}$ a holomorphic map such that $H(M \cap U)\subset M'.$ Then for $q\in M$ near $p:$
$H(\mathcal{S}_q)\subset \mathcal{S}'_{H(q)}.$
\end{itemize}
\end{proposition}

For a proof of the above and further details on Segre varieties and their properties, we refer the reader to Chapter 10 of \cite{BER1999}.

The Segre varieties capture higher-order reflection, and play a crucial role in finite jet determination and rigidity results; see e.g. \cite{BER1999},\cite{EbenfeltLamelZaitsev2003}.

\subsection{Automorphism group of a real-analytic hypersurface}
We write $f:\left(\mathbb{C}^N, p \right) \to ( \mathbb{C}^{N'},p')$ to denote a germ of a holomorphic map at $p \in \mathbb{C}^N$ with $f(p)=p'.$
\begin{definition}
    For a real-analytic hypersurface $M \subset \mathbb{C}^N$ we define the set of local  automorphisms at a point $p \in M$ as
    \begin{align*}
        \mathrm{Aut}(M,p)=\left\{f: (\mathbb{C}^N,p) \to (\mathbb{C}^N, p') \text{ holomorphic, } f(M) \subset M,  \det(J_f(p)) \neq 0 \right\},
    \end{align*}
    where $p'=f(p),$ and the stability or isotropy group of $M$ at $p\in M$ as
    \begin{align*}
        \mathrm{Aut}_p(M,p)=\left\{ f \in \mathrm{Aut}(M,p): f(p)=p \right\}.
    \end{align*}
\end{definition}

Closely related to the automorphisms of a real-analytic hypersurface are the infinitesimal CR automorphisms, namely: 
\begin{definition}
    If $M$ is a real-analytic hypersurface of $\mathbb{C}^N,$ the set of infinitesimal CR automorphisms of $M$ is defined as
$$\mathfrak{hol}(M,p)=\left\{ X=\sum_j a_j(z) \frac{\partial}{\partial z_j}: a_j \in \mathcal{O}_p, \operatorname{Re} X \text{ tangent to } M \right\},$$
where $\mathcal{O}_p$ denotes the ring of germs of holomorphic functions at $p.$ That is, the set of holomorphic vector fields whose real part is tangent to $M$ near $p.$ We also define the subalgebra $\mathfrak{hol}_0(M,p) \subset \mathfrak{hol}(M,p)$ as
\begin{align*}
    \mathfrak{hol}_0(M,p)=\left\{ X \in \mathfrak{hol}(M,p): X(p)=0 \right\}
\end{align*}
\end{definition}

The above sets are key tools for the finite jet determination of maps since the following holds:

\begin{proposition} \label{flowsinfi}  Let $M \subset \mathbb{C}^N$ be a real-analytic hypersurface. 
\begin{itemize}
    \item[(i)] The flows $\Phi_t^X(Z)$ of $X \in \mathfrak{hol}(M, p)$ satisfy $\Phi_t^X(Z) \in \mathrm{Aut}(M, p)$ for $t \in \mathbb{R}$ in a small neighborhood of $0.$
    \item[(ii)] Let $\mathrm{Aut}(M,p)$ be a finite dimensional Lie group. Then $\mathfrak{hol}(M,p)$ is the Lie algebra of $\mathrm{Aut}(M,p).$
\end{itemize}
\end{proposition}
For a proof of the above result, we refer the reader to Chapters 12.4 and 12.5 of \cite{BER1999}.

For our computations in the next section, we note that if $M \subset \mathbb{C}^N$ is a real-analytic hypersurface with defining function $\rho,$ $p \in M$ and $f \in \mathrm{Aut}(M,p),$ then it holds $\rho(f(z), \overline{f(z)})=0$ whenever $\rho(z, \bar{z})=0,$ by definition. Equivalently, the real-analytic function $\rho(f(z), \overline{f(z)})$ vanishes on the real-analytic hypersurface $\{\rho=0\}.$ Thus, there exists a real-analytic function $A$ such that $\rho (f(z), \overline{f(z)})=A(z,\bar{z}) \rho(z, \bar{z}).$

\subsection{Basic Properties of the Worm Domain}
This subsection contains some analytic and geometric properties of the worm, as well as a simplified version of the domain that is relevant for the main results of this paper. We begin with the following definitions:

\begin{definition}
    A pseudoconvex domain $\Omega$ is said to have a Stein neighborhood basis if there exists a sequence of smoothly bounded, pseudoconvex domains $U_1 \supset \supset U_2 \supset \supset \cdots \overline{\Omega}$ such that $\bigcap_{j=1}^{\infty}U_j=\overline{\Omega}.$ A domain failing this property is said to have a nontrivial Nebenhülle.
\end{definition}

\begin{definition}
    Let $\Omega \subset \mathbb{C}^N$ be a bounded domain. We say that $\Omega$ satisfies Condition $\mathrm{R}$ if the Bergman projection $P$ maps smooth functions on $\overline{\Omega}$ to smooth functions on $\overline{\Omega},$ that is, $P: \mathrm{C}^{\infty}(\overline{\Omega}) \to  \mathrm{C}^{\infty}(\overline{\Omega}).$
\end{definition}

\begin{theorem} The worm domain $\mathcal{W}$ satisfies the following properties:
    \begin{itemize}
    \item[(i)] $\mathcal{W}$ is a smoothly bounded, pseudoconvex domain.
    \item[(ii)] Its boundary is strongly pseudoconvex except along the annulus
    $$\mathcal{A} = \left\{ (z, 0) \in \partial \mathcal{W} : \left| \log |z|^2 \right| \leq \mu \right\} \subset \partial \mathcal{W}.$$
    \item[(iii)] It does not admit a Stein neighborhood basis when $\mu> \frac{\pi}{2}$; equivalently it has a nontrivial Nebenhülle when $\mu> \frac{\pi}{2}$.
    \item[(iv)] There exists no global plurisubharmonic defining function.
    \item[(v)] The worm fails to satisfy Condition $\mathrm{R}$.
\end{itemize}
\end{theorem}
A proof of $(i)-(iv)$ can be found e.g. in \cite{krantzpeloso} and $(v)$ is proven in \cite{Christ1996GlobalI}.

In order to simplify the analysis of the worm domain, many authors have used modified versions. In this paper, we consider the following unbounded version:
\begin{align*}
    \mathcal{W}_{\mathrm{un}}=\left\{ (z,w) \in \mathbb{C}^2: \left|w-e^{i \log{|z|^2}} \right|^2 <1 ,\text{ } z \neq 0 \right\}.
\end{align*}
It is easy to see that $\mathcal{W}_{\mathrm{un}}$ is an unbounded, pseudoconvex domain. Its boundary is non-smooth at points of the form $(0,w).$ The smooth part of $\partial \mathcal{W}_{\mathrm{un}},$ denoted by $\partial \mathcal{W}_{\mathrm{un}, \infty},$ however, is strongly pseudoconvex except along the annulus $\mathcal{A}_{\mathrm{un}}=\left\{ (z,0) \in \partial \mathcal{W}_{\mathrm{un}, \infty}, z \neq 0 \right\}.$ 

As a first step towards finding the automorphisms of the worm domain, we will work with the above simplified version and then use the results for $\mathcal{W}_{\mathrm{un}}$ to determine the automorphisms of the (bounded) worm domain $\mathcal{W}.$

\subsection{Boundary behavior of holomorphic functions on pseudoconvex domains}

Since the worm does not satisfy Condition $\mathrm{R}$, we need to rely on relevant local results. For the next theorem, we need the following local version of Condition $\mathrm{R}$ on smoothly bounded domains $\Omega \subset \mathbb{C}^N$:
\begin{definition}
    Let $\Omega \subset \mathbb{C}^N$ be a smoothly bounded pseudoconvex domain, and $p \in \partial \Omega.$ We say that $\Omega$ satisfies Local Condition $\mathrm{R}$ at $p$ if the Bergman projection $P$ associated to $\Omega$ maps smooth functions on $\overline{\Omega}$ into the space of functions in $L^2(\Omega)$ that are smooth up to the boundary in a neighborhood $U$ of $p,$ that is, $P: C^{\infty}(\overline{\Omega}) \to L^2(\Omega) \cap C^{\infty}(\overline{\Omega} \cap U)$
\end{definition}
The Local Condition $\mathrm{R}$ gives a local criterion for the smooth extension of holomorphic functions between smoothly bounded pseudoconvex domains in $\mathbb{C}^N$. In particular, by Bell \cite{bell1984local}:
\begin{theorem} \label{bellext}
    Let $f: \Omega_1 \to \Omega_2$ be a proper holomorphic mapping between smoothly bounded pseudoconvex domains in $\mathbb{C}^N.$ If $\Omega_1$ satisfies Local Condition $\mathrm{R}$ at a boundary point $p,$ then $f$ extends smoothly to $\partial \Omega_1$ near $p.$
\end{theorem}

\section{Automorphisms of the boundary of the unbounded worm domain}
We begin now the analysis of the simplified version of the worm domain introduced in the previous section. The first step is to determine the CR automorphisms of the smooth and strictly pseudoconvex part of its boundary. In particular:
\begin{proposition} \label{autbun}
    Consider the smooth and strongly pseudoconvex points of the boundary of the unbounded worm domain
    \begin{align*}
        \partial \mathcal{W}_{\mathrm{un}, \infty}\setminus \mathcal{A}_{\mathrm{un}}=\left\{(z,w) \in \mathbb{C}^* \times \mathbb{C}^*: \left|w-e^{i \log|z|^2} \right|^2=1 \right\}.
    \end{align*}
    The CR automorphisms of $\partial \mathcal{W}_{\mathrm{un},\infty} \setminus \mathcal{A}_{\mathrm{un}}$ are given by  $H(z,w)=( e^{\alpha+ i \beta}z, e^{2i \alpha} w),$ or $H(z,w)=\left(e^{\alpha+i \beta} \frac{2^i}{w^i z}, e^{2i \alpha}w \right)$ for $\alpha, \beta \in \mathbb{R},$ where $w^i$ denotes $e^{i \log w}.$
\end{proposition}

We note that the latter automorphism is well-defined on $\partial \mathcal{W}_{\mathrm{un}, \infty}\setminus \mathcal{A}_{\mathrm{un}}.$ Indeed, denote by $\log \zeta$ the principal branch of the logarithm for $\zeta \in \mathbb{C} \setminus (-\infty,0].$ Then the function $L(z,w)=\log(w e^{-i \log|z|^2})+i \log|z|^2$ provides a holomorphic logarithm of $w$ on $\partial \mathcal{W}_{\mathrm{un}, \infty}\setminus \mathcal{A}_{\mathrm{un}},$ by Lemma 1.2 of \cite{KrantzPelosoStoppato2016}. Indeed $w e^{-i \log|z|^2} \in \overline{ B(1,1)} \subset \mathbb{C},$ and points with $z=0$ or $w=0$ do not belong to $\partial \mathcal{W}_{\mathrm{un}, \infty} \setminus \mathcal{A}_{\mathrm{un}}.$ Consequently $\frac{1}{w^i z}=\frac{e^{-i \log w}}{z}$ is a well-defined function on $\partial \mathcal{W}_{\mathrm{un}, \infty} \setminus \mathcal{A}_{\mathrm{un}}.$

The setup for the proof is as follows. Let
\begin{align*}
    \rho(z,w)=\left|w- e^{i \log{|z|^2}}\right|^2-1=|w|^2-we^{-i \log{|z|^2}}-\bar{w}e^{i \log{|z|^2}}
\end{align*}
be a defining function for $\mathcal{W}_{\mathrm{un},\infty} \setminus \mathcal{A}_{\mathrm{un}}.$ Equivalently, in the complexified form, we have $\rho(z,w,\chi,\tau)=w \tau-w e^{-i \log{(z \chi)}} -\tau e^{i \log{(z \chi)}}.$
A CR automorphism $H=(f,g): \partial \mathcal{W}_{\mathrm{un},\infty} \setminus \mathcal{A}_{\mathrm{un}}  \to \partial \mathcal{W}_{\mathrm{un},\infty} \setminus \mathcal{A}_{\mathrm{un}}$ extends holomorphically in a neighbourhood of the strictly pseudoconvex part of the hypersurface $\partial \mathcal{W}_{\mathrm{un},\infty} \setminus \mathcal{A}_{\mathrm{un}}$ and since $\partial \mathcal{W}_{\mathrm{un},\infty} \setminus \mathcal{A}_{\mathrm{un}}$ is real-analytic, the map is also real-analytic.

Hence, the map satisfies the following complexified equation:
\begin{dmath} \label{1}
    g(z,w) \bar{g}(\chi,\tau)-g(z,w) e^{-i \log{\left(f(z,w) \bar{f}(\chi,\tau)\right)}}-\bar{g}(\chi,\tau) e^{i \log{\left( f(z,w) \bar{f}(\chi,\tau)\right)}}=0
\end{dmath}
for all $(z,w,\chi,\tau) \in \mathbb{C}^4$ satisfying $w \tau-w e^{-i \log{(z \chi)}} -\tau e^{i \log{(z \chi)}}=0.$

We work locally near the point $(1,2) \in \partial \mathcal{W}_{\mathrm{un},\infty}$ and start by determining all local CR automorphisms.
These maps can be classified into two types:
\begin{itemize}
    \item Isotropies $\mathrm{Aut}_{(1,2)}(\partial \mathcal{W}_{\mathrm{un},\infty},(1,2))$: CR maps $H=(f,g)$ with $H(1,2)=(1,2).$
    \item Non-isotropies: point-moving CR maps $H=(f,g).$
\end{itemize}
For computational simplicity, we use two different approaches for the isotropies and non-isotropies.

Many of the symbolic computations presented in this section were performed using the computer algebra system \textit{Mathematica} \footnote{The corresponding notebook is available at \url{https://phaidra.univie.ac.at/o:2345247}.}. The main results obtained with the aid of \textit{Mathematica} were independently verified by direct calculations. 

We begin by determining the local isotropies, or the stability group, of $\partial \mathcal{W}_{\mathrm{un},\infty} \setminus \mathcal{A}_{\mathrm{un}}.$ For this, we need the following lemma:

\begin{lemma} \label{coniso}
    Let $H=(f,g)\in \mathrm{Aut}_{(1,2)}\left(\partial \mathcal{W}_{\mathrm{un},\infty},(1,2)\right)$ be a non-constant map. Then it satisfies the following conditions:
    \begin{itemize}
    \item[(i)] $f_z(1,2) \neq 0$
    \item[(ii)] $g_w(1,2) \neq 0$
    \item[(iii)] $g_z(1,2)=0$
    \item[(iv)] $g_w(1,2)=\left|f_z(1,2)\right|^2 \in \mathbb{R}$
    \item[(v)] $g_{zz}(1,2)=-4f_z(1,2)^2+4\left|f_z(1,2)\right|^2 $
    \item[(vi)] $g_{zw}(1,2)=-4f_w(1,2)f_z(1,2)+4\overline{f_w(1,2)}f_z(1,2)+2i|f_z(1,2)|^2 \left(1-f_z(1,2) \right)$
    \item[(vii)] $f_{zz}(1,2)=\frac{8 f_z(1,2) \overline{f_w(1,2)}}{\overline{f_z(1,2)}}+(1-i) f_z(1,2)^2-(1-i) f_z(1,2)+4 f_w(1,2).$
    \end{itemize}
\end{lemma}
\begin{proof}
Since $H=(f,g)\in \mathrm{Aut}_{(1,2)}\left(\partial \mathcal{W}_{\mathrm{un},\infty},(1,2)\right),$ the remark at the end of Section 2.2 shows the existence of a real-analytic function $A$ such that, in complexified form, 
\begin{align*}
        B(z,w,\chi,\tau)&:=g(z,w) \bar{g}(\chi,\tau)-g(z,w) e^{-i \log{\left(f(z,w) \bar{f}(\chi,\tau)\right)}}-\bar{g}(\chi,\tau) e^{i \log{\left(f(z,w) \bar{f}(\chi,\tau)\right)}}\\
        &=A(z,w,\chi,\tau) \rho(z,w,\chi,\tau).
\end{align*}
Solving $\rho(z,w,\chi,\tau)=0$ for $w$ yields  $w_0=\frac{2z^{2i}\chi^{2i}}{2 \chi^i  z^i+\chi ^{2 i}-2 \chi ^i}.$ Substituting $w=w_0$ into the equation, the right-hand side vanishes. Therefore, expanding the left-hand side as a power series in $(z, \chi, \tau)$ near $(1,1,2)$, all coefficients must vanish. Equivalently, if
\begin{dmath} \label{4} 
B(z,w_0,\chi,\tau)=\sum_{l,j,k=1}^{\infty}B_{ljk} (z-1)^l (\chi-1)^j (\tau-2)^k,
\end{dmath}
then all $B_{ljk}=0.$ In particular, the equations $B_{ljk}=0,$ for $l,j,k=0,1,2$ impose the conditions:
\begin{align*}
    g_z(1,2)=0, \text{ } g_w(1,2)=\left|f_w(1,2)\right|^2 \in \mathbb{R}.
\end{align*}

Since $H \in \mathrm{Aut}_{(1,2)}( \partial \mathcal{W}_{\mathrm{un},\infty},(1,2))$ is non-constant, the Jacobian matrix $$\mathrm{Jac}_H(1,2)=\begin{bmatrix}
f_z(1,2) & f_w(1,2) \\
g_z(1,2) & g_w(1,2)
\end{bmatrix}$$
must have nonzero determinant. Given that $g_z(1,2)= 0,$ it follows that $f_z(1,2)$ and $g_w(1,2)$ are both nonzero.

Assuming $f_z(1,2), g_w(1,2) \neq 0,$ the remaining equations $B_{ljk}=0,$ for $ l,j,k=0,1,2$ yield the conditions (v)-(vii). This completes the proof.
\end{proof}
Using the above, we can now determine the local isotropies in a neighborhood of $(1,2) \in \partial \mathcal{W}_{\mathrm{un},\infty}.$ For the proof, we use techniques found, e.g., in \cite{BER1999} and \cite{EbenfeltLamelZaitsev2003}.
\begin{lemma} \label{autisobun}
    The group of local isotropies $\mathrm{Aut}_{(1,2)}\left(\partial \mathcal{W}_{\mathrm{un}, \infty},(1,2)\right)$ is an explicitly computable Lie group of real dimension $5.$
\end{lemma}

\begin{proof}
As a first step, we will compute all CR maps $H=(f,g)$ satisfying (\ref{1}). Let
$$\bar{L}=\rho_{\tau} \frac{\partial}{\partial \chi}-\rho_{\chi} \frac{\partial}{\partial \tau}=\left(w-e^{i \log{(z \chi)}} \right)\frac{\partial}{\partial \chi}-\left(\frac{iw}{\chi} e^{-i \log(z \chi)}-\frac{i \tau}{\chi} e^{i \log{(z\chi)}} \right)\frac{\partial}{\partial \tau}$$
be the basis of CR vector fields for $\partial \mathcal{W}_{\mathrm{un}}$ locally around the point $(1,2).$
Applying $\bar{L}$ to both sides of (\ref{1}), we obtain a system of equations

\begin{dmath}\label{2}
\begin{cases}
g(z,w) \bar{g}(\chi,\tau)
  - g(z,w) e^{-i \log{\left(f(z,w) \bar{f}(\chi,\tau)\right)}}
  - \bar{g}(\chi,\tau) e^{i \log{\left(f(z,w) \bar{f}(\chi,\tau)\right)}} = 0 \\
\left(g(z,w)-e^{i \log{\left(f(z,w) \bar{f}(\chi,\tau)\right)}} \right) \bar{L} \bar{g}(\chi,\tau)\\
  \hspace{1.3em}+ \left( i \frac{g(z,w)}{\bar{f}(\chi,\tau)} 
    e^{-i \log{\left(f(z,w) \bar{f}(\chi,\tau)\right)}}
  - i \frac{\bar{g}(\chi,\tau)}{\bar{f}(\chi,\tau)} 
    e^{i \log{\left(f(z,w) \bar{f}(\chi,\tau)\right)}} \right) \bar{L}\bar{f}(\chi,\tau) = 0
\end{cases}
\end{dmath}

We set $(\chi,\tau)=(1,2)$ and $(z,w)=\left(z, \frac{2 z^{2i}}{2 z^i-1} \right) \in \mathcal{S}_{(1,2)}.$ Since the point $(1,2)$ is $1-$nondegenerate, the system (\ref{2}) is solvable, and we obtain explicit formulas for the functions $f$ and $g$ along $\mathcal{S}_{(1,2)}.$

We also consider the vector fields
$$S_1= \frac{\partial}{\partial z}+ \left(\frac{2 i \frac{\tau}{z}  e^{2i \log{(z \chi)}}} {\tau  e^{i \log{(z \chi)}}-1}-\frac{i \frac{\tau^2}{z} e^{ 3i \log{(z \chi)}} }{\left(\tau e^{i \log{(z \chi)}}-1 \right)^2}\right) \frac{\partial}{\partial w}$$
and $$S_2=\frac{\partial}{\partial \tau}+ \frac{e^{2i \log{(z \chi)}}}{\left( \tau e^{i \log{(z \chi)}}-1\right)^2}  \frac{\partial}{\partial w}.$$
Applying $S_1$ and $S_2$ yields formulas for the first-order derivatives of $f$ and $g$ along $\mathcal{S}_{(1,2)}^1.$ Using the reflection principle, we also derive formulas for $\bar{f}$ and $\bar{g},$ and their derivatives along $\mathcal{S}_{(1,2)}^1.$ Similarly, we set in (\ref{2}) $$(z,w)=\left(z, \frac{2z^{2i}}{2 z^i \chi^i+ \chi^{2i}-2 \chi^i} \right) \text{} \text{ and} \text{ } (\chi,\tau)=\left(\chi, \frac{2}{-\chi^{2i}+2 \chi^{i}} \right),$$ using the convention $z^{ia}=e^{i a \log{z}}$ for $a \in \mathbb{R},$
and solve to obtain formulas for $f$ and $g$ along the second Segre variety $\mathcal{S}_{(1,2)}^2.$

Since the second Segre map $(z,\chi) \mapsto \left(z, \frac{2z^{2i}}{2 z^i \chi^i+ \chi^{2i}-2 \chi^i} \right) $ is generically of full rank, we can extend $f$ and $g$ to $\partial \mathcal{W}_{\mathrm{un}, \infty} \setminus \mathcal{A}_{\mathrm{un}}.$ To do so, we need to solve the equation $w=\frac{2z^{2i}}{2 z^i \chi^i+ \chi^{2i}-2 \chi^i}$ for $\chi.$
Writing 
\begin{align} \label{3}
    w= \frac{2 z^{2i}}{2 z^i} + \frac{4i (z^i-1) z^{2i}}{\left(2z-1 \right)^2} (\chi-1)+ O( \chi-1)^2,
\end{align}
we observe that the Implicit Function Theorem does not apply directly. Hence we use the proof of Proposition 2.11 in \cite{Baouendi:1997tl} and apply a desingularization method. Define new parameters:
\begin{align*}
    t:= \frac{w-\frac{2 z^{2i}}{2 z^i}}{\left(\frac{4i (z^i-1) z^{2i}}{\left(2 z^i-1 \right)^2} \right)^2 } \text{  and  } u:= \frac{\chi-1}{\left(\frac{4i (z^i-1) z^{2i}}{\left(2 z^i-1 \right)^2} \right)^2}.
\end{align*}
Then, equation $(\ref{3})$ becomes 
$t=u+O(u^2),$ and the Implicit Function Theorem applies. Thus, we write $u$ as a holomorphic function of $t$ and $z,$ $u:=\phi(z,t)$ as follows:
\begin{align*}
    u=-\frac{i \left(2 z^i-1\right)^2}{4z^{2i} \left(z^i-1 \right)} \left(\left(\frac{-8 t z^{4 i}+4 (t+2) z^{3 i}-4 (t+3) z^{2 i}+6 z^i-1}{8 (2 t+1) z^{3 i}-4 (2 t+3) z^{2 i}+6 z^i-1}\right)^{i}-1\right).
\end{align*}
Substituting $\chi=\left(\frac{4i (z^i-1) z^{2i}}{\left(2 z^i-1 \right)^2} \right)^2\phi(t,z)+1$ into the expressions for $f$ and $g$ along $\mathcal{S}_{(1,2)}^2,$ and applying \Cref{coniso}, we obtain explicit formulas for candidate functions $f$ and $g$ in terms of $t$ and $z.$ For simplicity, we continue to denote these functions as $f(z,t)$ and $g(z,t),$ and expand as power series:
\begin{align} \label{5}
    f(z,t)=\sum_{j=0}^{\infty}a_j(z)t^j \text{ } \text{ and } \text{ } g(z,t)=\sum_{j=0}^{\infty}b_j(z)t^j.
\end{align}
Since $$t= \frac{w-\frac{2 z^{2i}}{2 z^i}}{\left(\frac{4i (z^i-1) z^{2i}}{\left(2z^i-1 \right)^2} \right)^2 }=:\frac{p(z,w)}{(z-1)^2},$$ where $p$ is holomorphic near $(1,2),$ we require each $a_j(z)$ and $b_j(z)$ to be divisible by $(z-1)^{2j}.$ By the Weierstrass Division Theorem, we write:
\begin{align*}
    a_j(z)=(z-1)^{2j}p^a_j(z)+r^a_j(z) \text{ and } b_j(z)=(z-1)^{2j}p^b_j(z)+r^b_j(z),
\end{align*}
where the remainders $r_j^{a}(z)$ and $r_j^{b}(z)$ are polynomials of degree less than $2j.$ To preserve holomorphicity of $f$ and $g,$ we must have $r^a_j(z)=r^b_j(z)=0$ for all $j.$ That is, writing $a_j(z)=\sum_{k=0}^{\infty} a_{j,k} (z-1)^k$ and $b_j(z)=\sum_{k=0}^{\infty} b_{j,k}(z-1)^k,$ we conclude $a_{j,k}=b_{j,k}=0,$ for all $ k=0, \cdots, 2j-1.$

In particular, for $j=3,$ we obtain 
\begin{dmath*}
    a_{3,0}=-\frac{4096 f_w(1,2)}{|f_z(1,2)|^2 f_z(1,2)} \left(f_w(1,2) \left(g_{ww}(1,2)-|f_z(1,2)|^4 \right)\\ +(3 i+1) |f_z(1,2)|^2 f_w(1,2)^2-f_{ww}(1,2) |f_z(1,2)|^2+4 f_w(1,2)^3\right).
\end{dmath*}
This yields two possibilities: either $f_w(1,2)=0,$ or 
\begin{dmath*}
    g_{ww}(1,2)= \frac{1}{f_w(1,2)} \left(-(1+3 i) |f_z(1,2)|^2 f_w(1,2)^2+ |f_z(1,2)|^4 f_w(1,2) \\ +f_{ww}(1,2) |f_z(1,2)|^2-4 f_w(1,2)^3 \right).
\end{dmath*}

We first consider the latter case, where $f_w(1,2) \neq 0,$ and use the condition for $g_{ww}(1,2)$ to obtain holomorphic expressions for $f$ and $g.$ Expanding in powers of $z-1$ we write
\begin{align*}
        f(z,w)=\sum_{j=0}^{\infty}f_j(w) (z-1)^j \text{ and } g(z,w)=\sum_{j=0}^{\infty}g_j(w) (z-1)^j.
\end{align*}
Substituting into equation (\ref{4}) and imposing the vanishing of the coefficients $B_{ljk}=0$ for $l,j,k=1,2,$ we obtain
\begin{dmath*}
\overline{f_{zw}(1,2)}= \frac{1}{|f_w(1,2)|^2 f_z(1,2)} \left(4 |f_w(1,2)|^4 +2 i |f_w(1,2)|^2 f_w(1,2)\overline{f_z(1,2)}\\-2 i |f_w(1,2)|^2 \overline{f_w(1,2)} f_z(1,2) +2 |f_z(1,2)|^2 |f_w(1,2)|^2 \\ +f_w(1,2) |f_z(1,2)|^2 \overline{f_{ww}(1,2)} -f_{zw}(1,2) |f_w(1,2)|^2 \overline{f_z(1,2)}\\ +f_{ww}(1,2) |f_z(1,2)|^2 \overline{f_w(1,2)}\right).
\end{dmath*}

Let us first consider the case $f_w(1,2)=:\xi+i \rho \neq 0$ and $\operatorname{Im}f_z(1,2)=:\nu \neq 0.$ Using the above conditions, we obtain the first parameterization for the automorphisms  $H_1(z,w)=\left(f_1(z,w),g_1(z,w) \right)$ as follows:
\begin{align*}
    f_1(z,w)=\left(\frac{N_{f_1}(z,w)}{D_{f_1}(z,w)} \right)^i \text{ } \text{ and } \text{ } g_1(z,w)=\frac{N_{g_1}(z,w)}{D_{g_1}(z,w)}
\end{align*}
where

\begin{dmath*}
    N_{f_1}(z,w)=4 i z^{2i} \left(\mu ^2 (\xi^2+ \rho^2  +\xi \sigma+\rho\varphi)-\mu  (\xi +i \rho ) (\nu  (\rho+\varphi+ i( \xi+ \sigma) )\\
    +(\xi -i \rho ) (2 \rho +\psi ))+i \left(\nu^2( \xi \varphi-\rho \sigma)+( \xi^2+ \rho^2)(2 \nu+ \psi)-2i ( \xi^2+ \rho^2)^2 \right)\right)\\
    -4 i z^i w \left(\mu ^2 (\xi^2+\rho^2+ \xi\sigma +\rho\varphi ))-\mu  (\xi +i \rho ) (\nu  ( \rho+\varphi +i ( \xi+ \sigma) )\\
    +(\xi -i \rho ) (2 \rho +\psi ))+i \nu ^2 (\xi  \varphi -\rho  \sigma )+\nu  (\xi ^2+\rho ^2) (\xi +i (\rho +\psi ))+2 (\xi ^2+\rho ^2)^2\right)\\
    +i w \left(2 \mu ^2 (\xi^2+ \rho^2+ \xi\sigma+\rho\varphi)-\mu  (\xi +i \rho ) (\nu  (\rho+2 \varphi+ i (\xi +2 \sigma)  )\\
    +2 (\xi -i \rho ) (2 \rho +\psi ))+\nu ^2 \left(\xi ^2 +\rho ^2+2 i (\xi  \varphi- \rho  \sigma) \right)+2 \nu  \left(\xi ^2+\rho ^2\right) (2 \xi +i \psi )\\
    +4 \left(\xi ^2+\rho ^2\right)^2\right)
\end{dmath*}
\begin{dmath*}
    D_{f_1}(z,w)=4 z^{2i} \left(i \mu ^2 (\xi^2+ \rho^2+\xi\sigma +\rho \varphi )+\mu  (\xi +i \rho ) \left(\nu  \left(\xi+ \sigma+i (\xi ^2+ \rho^2 -\rho -\varphi) \right)\\
    -i (\xi -i \rho ) (2 \rho +\psi )\right)+\nu ^2 \left( (\xi^2+ \rho^2)(\xi+i \rho)+\rho  \sigma -\xi  \varphi \right)-\nu  (\xi ^2+\rho ^2) (2 \rho +\psi )\\
    +2 i (\xi ^2+\rho ^2)^2\right)-z^i w \left(\mu ^2 \left((\nu+4i) (\xi^2+ \rho^2)+4i(\xi \sigma+ \rho \varphi) \right) \\+4 i \mu  (\xi +i \rho )(\nu  (\xi^2+ \rho^2 -\rho-\varphi-i( \xi+\sigma) )-(\xi -i \rho ) (2 \rho +\psi ))+\nu ^3 \left(\xi ^2+\rho ^2\right)\\
    +4 \nu ^2 \left((\xi^2+\rho^2)(\xi+i \rho)-\xi \varphi+ \rho \sigma \right)+4 i \nu  \left(\xi ^2+\rho ^2\right) (\xi +i (\rho +\psi ))\\
    +8 i \left(\xi ^2+\rho ^2\right)^2\right) +w \left(\mu ^2 \left(\nu (\xi^2+ \rho^2)+2i (\xi^2+ \rho^2+\xi \sigma+ \rho \varphi) \right)+\mu  (\xi +i \rho ) \left(\nu  \left(\xi+2 \sigma+ 2i (\xi^2+ \rho^2)-i (\rho+2 \varphi)\right)-2 i (\xi -i \rho ) (2 \rho +\psi )\right)\\
    +\nu ^3 (\xi ^2+\rho ^2)+\nu ^2 \left((\xi^2+\rho^2)(2\xi+2i \rho+i)+2( \rho \sigma-\xi \varphi)\right)+2 i \nu  \left(\xi ^2+\rho ^2\right) (2 \xi +i \psi )+4 i \left(\xi ^2+\rho ^2\right)^2\right)
\end{dmath*}
\begin{dmath*}
    N_{g_1}(z,w)= -2 i \left(4 z^{2i} \left(\mu ^2 (\xi^2+ \rho^2+ \xi \sigma +\rho\varphi)+\mu  (\xi +i \rho ) (\nu  (\xi^2+ \rho^2-\rho-\varphi-i(\xi+ \sigma) )\\
    -(\xi -i \rho ) (2 \rho +\psi ))-i \nu ^2 \left((\xi^2+ \rho^2)(\xi+ i \rho)+\rho  \sigma -\xi  \varphi \right)
    \\
    +i \nu  \left(\xi ^2+\rho ^2\right) (2 \rho +\psi )
    +2 \left(\xi ^2+\rho ^2\right)^2\right)+z^i w \left(\mu ^2 \left((\xi^2+ \rho^2)(i \nu-4)\\
    -4( \xi \sigma+ \rho \varphi) \right)-4 \mu  (\xi +i \rho ) (\nu  (\xi^2+ \rho^2-\rho-\varphi-i(\xi+ \sigma)  )\\
    -(\xi -i \rho ) (2 \rho +\psi )) +i \nu ^3 (\xi ^2+\rho ^2)+4 i \nu ^2 \left((\xi^2+ \rho^2)(\xi+i \rho)+\rho  \sigma-\xi  \varphi  \right)-4 \nu  \left(\xi ^2+\rho ^2\right) (\xi +i (\rho +\psi ))-8 \left(\xi ^2+\rho ^2\right)^2\right)+w \left(\mu ^2 \left((\xi^2+ \rho^2)(2-i \nu)\\
    +2 (\xi \sigma+ \rho \varphi) \right)+\mu  (\xi +i \rho ) (\nu  (2( \xi^2+ \rho^2)-\rho-2 \varphi-i( \xi+ 2 \sigma) )\\
    -2 (\xi -i \rho ) (2 \rho +\psi ))-i \nu ^3 \left(\xi ^2+\rho ^2\right)+\nu ^2 \left((\xi^2+ \rho^2)(2\rho+1-2i \xi)\\
    -2i( \xi \varphi+ \rho \sigma)\right)+2 \nu  \left(\xi ^2+\rho ^2\right) (2 \xi +i \psi )+4 \left(\xi ^2+\rho ^2\right)^2\right)\right)^2
\end{dmath*}
\begin{dmath*}
    D_{g_1}(z,w)=\left( 4 z^{2i} \left(-i \mu ^2 (\xi^2+ \rho^2+\xi \sigma+\rho \varphi )+\mu  (\xi +i \rho ) (-\nu(  \xi+ \sigma) +i \nu  (\rho +\varphi )\\
    +(\rho +i \xi ) (2 \rho +\psi ))+\nu^2(\xi \varphi-\sigma \rho)+ \nu(\xi^2+ \rho^2)(2 \rho+ \psi)-2i (\xi^2+ \rho^2)^2 \right)\\
    +4 i z^i w \left(\mu ^2 (\xi^2+ \rho^2+ \xi\sigma  +\rho \varphi )-\mu  (\xi +i \rho ) (\nu  (\rho+ \varphi +i( \xi +\sigma)  ) \\
    +(\xi -i \rho ) (2 \rho +\psi ))+i \nu ^2 (\xi  \varphi -\rho  \sigma )+\nu  \left(\xi ^2+\rho ^2\right) (\xi +i (\rho +\psi )) +2(\xi ^2+\rho ^2)^2\right)\\
    -i w \left(2 \mu ^2 (\xi^2+ \rho^2+ \xi \sigma + \rho \varphi )-\mu  (\xi +i \rho ) (\nu  (\rho+2 \varphi +i (\xi  +2 \sigma) ) \\
    +2 (\xi -i \rho ) (2 \rho +\psi ))+\nu ^2 \left(\xi ^2+\rho^2+2 i (\xi  \varphi -\rho  \sigma) \right) +2 \nu  \left(\xi ^2+\rho ^2\right) (2 \xi +i \psi )\\
    +4 \left(\xi ^2+\rho ^2\right)^2\right) \right) \left( 2 z^{2i} \left(\mu^2(\nu^2+2+i \mu\nu) (\xi ^2+\rho ^2)+2\mu ^2(\varphi \rho+ \xi  \sigma)+\mu  (\xi +i \rho ) \left(\nu ^3 (\rho +i \xi )+2\nu  \left(2 (\xi ^2+ \rho^2)- (\rho+ \varphi)- i (\xi+\sigma)  \right)\\-2 (\xi -i \rho ) (2 \rho +\psi )\right)+\nu ^4 (\xi ^2+\rho ^2)-2 i \nu ^2 \left((\xi^2+ \rho^2)(\xi+i \rho) +\rho  \sigma-\xi  \varphi \right)\\
    +2 i \nu (\xi ^2+\rho ^2) (2 \rho +\psi )+4(\xi ^2+\rho ^2\right)^2)+2 z^i w \left(-i \mu ^3 \nu  (\xi ^2+\rho ^2)\\
    -\mu ^2 \left((\xi^2+ \rho^2)(\nu^2-i \nu+2)+2( \xi \sigma+ \rho \varphi) \right)+\mu  (\xi +i \rho ) \left(-\nu ^3 (\rho +i \xi )\\
    -2 \nu  \left(2(\xi^2+ \rho^2)-\rho- \varphi-i(\xi+ \sigma) \right)+2 (\xi -i \rho ) (2 \rho +\psi )\right)-\nu^4(\xi^2+\rho^2)\\
    +i\nu^3(\xi ^2+\rho ^2)+2 i \nu ^2 \left((\xi^2+ \rho^2)(\xi+i \rho) +\rho  \sigma-\xi  \varphi \right)-2 \nu  (\xi ^2+\rho ^2) (\xi +i\rho \\
    +i\psi )-4 (\xi ^2+\rho ^2)^2\right)+w \left(i \mu ^3 \nu (\xi ^2+\rho ^2)+\mu ^2 \left((\xi^2+ \rho^2)(\nu^2-i \nu+2)\\
    +2 (\xi \sigma+ \rho \varphi) \right)+\mu  (\xi +i \rho ) \left(\nu ^3 (\rho +i \xi )+\nu  (4(\xi^2+ \rho^2)-(\rho+2 \varphi)-i( \xi+ 2 \sigma) )\\
    -2 (\xi -i \rho ) (2 \rho +\psi )\right)+\nu ^4 (\xi ^2+\rho ^2)-2 i \nu ^3 (\xi ^2+\rho ^2)+\nu ^2 \left((\xi^2+ \rho^2)(4 \rho+1\\
    -4i \xi)+2i( \xi \varphi- \rho \sigma) \right)+2 \nu  (\xi ^2+\rho ^2) (2 \xi +i \psi )+4 (\xi ^2+\rho ^2)^2\right) \right)
\end{dmath*}
for parameters $\mu,\nu, \xi, \rho, \sigma, \phi, \psi \in \mathbb{R}.$ The above maps do not belong to the connected component of the identity in the group of CR automorphisms of $\partial \mathcal{W}_{\mathrm{un}, \infty} \setminus \mathcal{A}_{\mathrm{un}}$, since they satisfy $f_w(1,2)=:\xi+i \rho \neq 0$ and $\operatorname{I m}f_z(1,2)=:\nu \neq 0.$

Assume now that $f_w(1,2)=\xi+ i \rho \neq 0$ and $\nu=\operatorname{Im}f_z(1,2)=0.$ Then, we obtain the following parametrization for the map which we denote by $H_2(z,w)=\left(f_2(z,w),g_2(z,w) \right):$
$$f_2(z,w)=\left(\frac{N_{f_2}(z,w)}{D_{f_2}(z,w)} \right)^i \text{ } \text{ and } \text{ } g_2(z,w)=\frac{N_{g_2}(z,w)}{D_{g_2}(z,w)}$$
\begin{dmath*}
    N_{f_2}(z,w)=4 z^{2i} \left(-i \mu ^2 (\rho  \sigma -\xi  \varphi )+i \mu  (\xi ^2+\rho ^2) (2 \xi -\psi )+2 (\xi ^2+\rho ^2)^2\right)\\
    +z^i w \left(4 i \mu ^2 (\rho  \sigma-\xi \varphi) +4 \mu  (\xi ^2+\rho ^2) (\rho +i (\psi-\xi) )-8 \left(\xi ^2+\rho ^2\right)^2\right)\\
    +w \left(\mu ^2 \left(\xi ^2+\rho^2+2 i (\xi  \varphi-\rho \sigma)\right)-2 \mu  \left(\xi ^2+\rho ^2\right) (2 \rho +i \psi )+4 \left(\xi ^2+\rho ^2\right)^2 \right)
\end{dmath*}
\begin{dmath*}
    D_{f_2}(z,w)=4 z^{2i} \left(-i \mu ^2 \left((\xi^2+ \rho^2)(\xi+i \rho)+\rho  \sigma-\xi  \varphi \right)+i \mu  (\xi ^2+\rho ^2) (2 \xi -\psi )\\
    +2 (\xi ^2+\rho ^2)^2\right)+z^i w \left(\mu ^3 \left(\xi ^2+\rho ^2\right)+4 \mu^2\left((\xi^2+ \rho^2)(i \xi-\rho)+i ( \rho \sigma-\xi \varphi) \right)\\
    + 4 \mu  (\xi ^2+\rho ^2) (\rho +i (\psi-\xi) )-8 (\xi ^2+\rho ^2)^2\right)+w\left(-\mu ^3(\xi ^2+\rho ^2)\\
    +\mu ^2 \left((\xi^2+ \rho^2)(2\rho+1-2i \xi)+2i( \xi \varphi-\rho \sigma) \right)-2 \mu  \left(\xi ^2+\rho ^2\right) (2 \rho +i \psi )\\
    +4 \left(\xi ^2+\rho ^2\right)^2 \right)
\end{dmath*}
\begin{dmath*}
    N_{g_2}(z,w)=-2 i \left(4 z^{2i} \left(-i \mu ^2 \left( (\xi^2+ \rho^2)(\xi+i \rho)+\rho  \sigma-\xi  \varphi \right)+i \mu (\xi ^2+\rho ^2) (2 \xi -\psi )\\
    +2 (\xi ^2+\rho ^2)^2\right)+z^i w \left(\mu ^3 (\xi ^2+\rho ^2)+4\mu^2\left((\xi^2+ \rho^2)(i \xi-\rho)+i(\rho \sigma-\xi \varphi) \right) \\
    +4 \mu  (\xi ^2+\rho ^2) (\rho +i (\psi-\xi) )-8 (\xi ^2+\rho ^2)^2\right)+w \left(-\mu ^3 (\xi ^2+\rho ^2)\\
    +\mu^2\left((\xi^2+ \rho^2)(2\rho+1-2i \xi)+2i(\xi \varphi-\rho \sigma) \right) -2 \mu (\xi ^2+\rho ^2)(2\rho+i\psi)\\
    +4 (\xi ^2+\rho ^2)^2\right)\right)^2
\end{dmath*}
\begin{dmath*}
    D_{g_2}(z,w)=\left(-2 z^{2i} \left(\mu ^4 (\xi ^2+\rho ^2)-2 i \mu ^2 \left((\xi^2+ \rho^2)(\xi+i \rho)+\rho  \sigma -\xi  \varphi \right)\\
    +2 i \mu  \left(\xi ^2+\rho ^2\right) (2 \xi -\psi )+4 (\xi ^2+\rho ^2)^2\right)+z^i w \left(2 \mu ^4 (\xi ^2+\rho ^2)-2 \mu ^3 (\xi ^2+\rho ^2)\\
    +4 \mu^2\left(2(\xi^2+ \rho^2)(\rho-i \xi)+i ( \xi \varphi- \rho \sigma)\right)+4 i \mu  (\xi ^2+\rho ^2) (\xi -\psi+i \rho )\\
    +8 (\xi ^2+\rho ^2)^2\right)-w \left(\mu ^4 (\xi ^2+\rho ^2)-2 \mu ^3 (\xi ^2+\rho ^2)+\mu^2\left((\xi^2+ \rho^2)(4 \rho+1-4i \xi) \\
    +2i( \xi \varphi-\rho \sigma) \right)  -2 i \mu  \psi  (\xi ^2+\rho ^2)+4 \mu  \rho  (\xi ^2+\rho ^2)-4 (\xi ^2+\rho ^2)^2\right)\right)\\
    \left(-4 z^{2i} \left(\mu ^2 (\xi  \varphi -\rho  \sigma )+\mu (\xi ^2+\rho ^2) (2 \xi -\psi )-2 i (\xi ^2+\rho ^2)^2\right)\\
    +z^i w \left(4 \mu ^2 (\xi  \varphi - \rho  \sigma)+4 \mu  (\xi ^2+\rho ^2) (\xi  -\psi+i \rho )-8 i (\xi ^2+\rho ^2)^2\right)\\
    +w \left(\mu ^2 \left(2( \rho \sigma- \xi \varphi)+i( \xi^2+ \rho^2) \right) +2 \mu (\xi ^2+\rho ^2)(\psi-2i \rho)+4 i (\xi ^2+\rho ^2)^2\right)\right)
\end{dmath*}
for real parameters $\mu,\xi, \rho, \sigma,\phi,\psi.$ The maps $H_2(z,w)|_{\mu=1,\sigma=\phi=\psi=0}$ with free parameters $\xi, \rho \in \mathbb{R}$ belong to the connected component of the identity in the group $\mathrm{Aut} \left( \partial \mathcal{W}_{\mathrm{un}, \infty},(1,2) \right).$ A straightforward computation using \Cref{flowsinfi} shows that only these correspond to infinitesimal CR automorphisms.

We now consider the second case, namely when $f_w(1, 2):=\xi+i \rho= 0.$ Using equation (\ref{5}), we find $b_{3,0}=4096f_{ww}(1,2) \overline{f_z(1,2)},$ which implies $f_{ww}(1,2)=0.$ 
We substitute this into the expressions for $f$ and $g,$ and expand as power series in $z$ near $1:$
\begin{align*}
        f(z,w)=\sum_{j=0}^{\infty}f_j(w) (z-1)^j \text{ and } g(z,w)=\sum_{j=0}^{\infty}g_j(w) (z-1)^j.
\end{align*}
Under these assumptions, equation (\ref{4}) yields the additional condition:
$$\operatorname{Re} \left(g_{ww}(1,2) \right)= |f_z(1,2)|^4 -|f_z(1,2)|^2.$$

Let us first consider the case $\operatorname{Im}f_z(1,2)=:\nu \neq 0.$ Applying the above conditions, we obtain the following formulas for which we use the notation $H_3(z,w)=\left(f_3(z,w),g_3(z,w) \right):$
\begin{dmath*}
    f_3(z,w)=\left(\left(w \left(2 i \psi(2z^{2i}-z^i+1)+\nu(\mu ^2 +\nu ^2)\right)\right)/ \\
    \left(4 i \psi  z^{2 i}+z^i w \left(\nu(\mu^2+\nu^2)(\mu+i \nu)-4i \psi \right) \\
    + w \left(\nu(\mu^2+\nu^2)(1-\mu-i\nu)+2i \psi \right)\right)\right)^i
\end{dmath*}
and
\begin{dmath*}
    g_3(z,w)=\left(2 \left(4i z^{2i} \psi+z^i \left( \nu(\mu^2+ \nu^2)(\mu+i \nu)-4i \psi \right)+ \nu(\mu^2+ \nu^2)(1-\mu- i \nu)\\
    +2i \psi \right)^2 \right)/ \left(\left(2z^{2i} \left(  \nu(\mu^2+\nu^2)^2+2i \psi \right)+2 z^i w \left(\nu(\mu^2+\nu^2)(\mu+i \nu)-2i \psi\\
    -\nu(\mu^2+\nu^2)^2  \right)+w \left( \nu(\mu^2+\nu^2)^2+\nu(\mu^2+\nu^2)(1-2(\mu+i\nu))+2i \psi \right)\right)\\
    \left(4 i \psi (  z^{2 i}-z^i w )+ w \left(\mu ^2+\nu ^2 +2 i \psi\right)\right) \right)
\end{dmath*}
for real parameters $\mu,\nu, \psi \in \mathbb{R}.$ The maps $H_3(z,w)|_{\mu=1, \psi=0}$ with parameter $\nu \in \mathbb{R}$ belong to the connected component of the identity in the group $\mathrm{Aut} \left( \partial \mathcal{W}_{\mathrm{un}, \infty},(1,2) \right),$ following the same argument as for $H_2.$

On the other hand, if $\operatorname{Im}f_z(1,2)=\nu =0,$ we obtain the following parametrization for $H_4(z,w)=\left( f_4(z,w),g_4(z,w)\right):$
\begin{dmath*}
    f_4(z,w)=\left(\frac{4i\psi(z^{2i}-z^iw)+w(\mu^2+2i \psi)}{4i \psi(z^{2i}-z^i w)+z^i w \mu^3+w \left(\mu^2(\mu-1)+2i \psi \right)} \right)^i
\end{dmath*}
\begin{dmath*}
    g_4(z,w)=\left(2 \left(4i \psi (z^{2i}-z^i w)+z^iw \mu^3+ w\left(\mu^2(\mu-1)+2i \psi \right) \right)^2\right)/ \\
    \left(\left(4i \psi(z^{2i}-z^i w)+w (\mu^2+2i \psi) \right) \left(2z^{2i} \left(\mu^4+ 2i \psi \right)\\
    +2z^i w \left(\mu^3(\mu-1)-2i \psi \right)
    +w \left(\mu^2(\mu-1)^2+2i \psi \right) \right)\right)
\end{dmath*}
for real parameters $\mu,\psi \in \mathbb{R}.$ Similarly, the maps $H_4$ belong to the connected component of the identity in $\mathrm{Aut} \left( \partial \mathcal{W}_{\mathrm{un}, \infty},(1,2) \right),$ as can be verified, in the same manner as for $H_2$ and $H_3.$

Therefore, we obtain a Lie group of real dimension $5,$ which completes the proof.
\end{proof}

The next step is to determine the second type of point-moving maps near $(1,2)\in \partial \mathcal{W}_{\mathrm{un}, \infty}.$ 

\begin{lemma} \label{autnonisobun}
    The factor group of local point-moving maps $$\mathrm{Aut}_{\mathrm{pm}}(\partial \mathcal{W}_{\mathrm{un},\infty},(1,2)):=\mathrm{Aut}\left(\partial \mathcal{W}_{\mathrm{un}, \infty} \setminus \mathcal{A}_{\mathrm{un}} ,(1,2)\right)/ \mathrm{Aut}_{(1,2)}\left(\partial \mathcal{W}_{\mathrm{un}, \infty},(1,2)\right)$$ is isomorphic to the group consisting of the following maps:
    \begin{align*}
        H_{5}(z,w)=\left(e^{\alpha+i\beta}z, \frac{e^{2i\alpha}z^{2i}w}{z^{2i}-\gamma iw} \right), \text{ } \text{ for } \alpha, \beta, \gamma \in \mathbb{R}.
    \end{align*}
\end{lemma}
\begin{proof}
Clearly the above maps are point-moving CR automorphisms, and $$\mathrm{Aut}_{\mathrm{pm}}(\partial \mathcal{W}_{\mathrm{un},\infty},(1,2)) \leq \dim_{\mathbb{R}} \mathrm{Aut} \left( \partial \mathcal{W}_{\mathrm{un}, \infty} \setminus \mathcal{A}_{\mathrm{un}} \right) =3.$$ 
\end{proof}

 Using the above, we can now prove \Cref{autbun}:
\begin{proof}[Proof of \Cref{autbun}]    
From \Cref{autisobun} and \Cref{autnonisobun}
we have formulas for all local CR automorphisms in a neighborhood of $(1,2) \in \partial \mathcal{W}_{\mathrm{un}, \infty} \setminus \mathcal{A}_{\mathrm{un}}. $
We now need to determine which of these local automorphisms $H_j(z,w),$ $j=1,\cdots,5$ extend globally. We set $z=e^{\zeta},$ with $\zeta \in \mathbb{C},$ and check, with a slight abuse of notation, whether $$H_j(\zeta,w)=H_j(\zeta+ 2i \pi n,w)$$ for all $n \in \mathbb{Z}$ and $j=1, \cdots ,5.$ That is, we check invariance under logarithmic monodromy.

First, consider $H_1$ for which we assumed that $(\xi, \rho) \neq (0,0)$ and $\nu \neq 0.$ Setting $(\zeta,w)=(0,2),$ the monodromy condition implies $g_1(2 i \pi n,2)=g_1(0,2)=2$ for every $n \in \mathbb{Z}.$ Taking limit as $n \to -\infty,$ the explicit formula for $g_1$ from \Cref{autisobun} yields $\xi=0$ and $\psi=\frac{2 \mu \rho^2}{ \mu ^2+\nu ^2}+ \frac{\left(\nu \sigma - 2 \rho^2 \right)+ \mu (\rho+\phi)}{\rho}.$ Substituting this relation into $g_1(2 \pi i n,2)=2$ and setting $n=1$ yields $\nu=0,$ a contradiction. Hence $H_1$ does not extend globally.

Similarly, the assumptions for $H_2$ are $(\xi, \rho) \neq (0,0)$ and $\nu=0.$ Since $f_z(1,2) =\mu+ i \nu,$ \Cref{coniso} implies $\mu \neq 0.$ Similarly with $H_1,$ taking the limit $n \to -\infty$ at $g_2(2 i \pi n,2)=g_2(0,2)=2$ yields $\xi=0$ and $\sigma=-\frac{ \rho \psi}{\mu}.$ Then $g_{2w}(2\pi i n,2)=g_{2w}(0,2)$ for all $n \in \mathbb{Z}$ yields either $\rho=0, \mu=1,$ a contradiction to $(\xi, \rho) \neq (0,0),$ or $\mu=-1$ and $\rho=-\frac{1}{2}.$ The latter condition yields the automorphism $H_2(z,w)=\left( \frac{2^i}{w^i z}, w \right).$ The remark following \Cref{autbun} shows that the above is a well-defined holomorphic map on $\partial \mathcal{W}_{\mathrm{un}, \infty} \setminus \mathcal{A}_{\mathrm{un}}.$  Moreover, direct calculation yields that $H_2 \circ H_2$ is the identity map.

For $H_3(z,w),$ we assume $\nu \neq 0.$ Setting $\zeta=0$ and $w=2$ the monodromy condition $g_3(2 i \pi,2)=g_3(0,2)=2$ does not give any solutions. Hence, $H_3$ does not extend globally.

Regarding $H_4,$ we have the assumption $\nu=0,$ which as in $H_2$ implies $\mu \neq 0.$ Setting $(\zeta,w)=(0,2),$ the monodromy condition $g_4(2\pi i,2)=g_4(0,2)=2,$ yields $\psi=0.$ Then considering $g_{4w}(2\pi i,2)=g_{4w}(0,2)$ yields either $\mu=1$, which gives the identity map, or $\mu=\frac{e^{2\pi}+1}{e^{2\pi}-1}.$ In the last case, the explicit formula for $f_4$ from \Cref{autisobun} yields $
f_4(2\pi i,2)=e^{\pi}\neq f_4(0,2)= 1,$
a contradiction.

Finally, consider $H_5.$ Setting $(\zeta,w)=(0,2)$ the monodromy condition $g_5(2\pi i,2)=g_5(0,2)=2$ yields $\gamma=0.$ Clearly, the maps $H_5$ with $\gamma=0$ extend globally for all $\alpha, \beta \in \mathbb{R}.$

Consequently, the only non-trivial automorphisms that extend globally are 
$$H(z,w)=\left( e^{\alpha+i \beta}z, e^{2i \alpha}w \right)$$
and
$$H(z,w)=\left( \frac{2^ie^{\alpha+i \beta}}{w^i z}, e^{2i \alpha}w \right)$$
for $\alpha, \beta \in \mathbb{R}.$ 
\end{proof}

From \Cref{autisobun} and \Cref{autnonisobun}, we can derive the following:
\begin{corollary} \label{locsph}
    The set $\partial \mathcal{W}_{\mathrm{un}, \infty}\setminus \mathcal{A}_{\mathrm{un}}$ is locally spherical; that is, the boundary of the unbounded worm domain is locally spherical near all smooth strongly pseudoconvex points. In particular, the mapping
    $$(z,w) \mapsto \left(i-iz^i,i-2iz^i+\frac{2iz^{2i}}{w} \right)$$
    maps locally $(\partial \mathcal{W}_{\mathrm{un}, \infty} \setminus \mathcal{A}_{\mathrm{un}},(1,2))$ into the Heisenberg hypersurface $(\mathbb{H},(0,0)).$
\end{corollary}
\begin{proof}
    Using \Cref{autisobun} and \Cref{autnonisobun},
    $$\dim_{\mathbb{R}} \mathfrak{hol}\left(\partial \mathcal{W}_{\mathrm{un},\infty} \setminus \mathcal{A}_{\mathrm{un}},(1,2) \right)=8=\dim_{\mathbb{R}} \mathfrak{hol}\left(\mathbb{H},(0,0)\right).$$
    Hence, by Chapter 2 of \cite{Jacobowitz1990}, or by the more general result \cite{KossovskiyShafikov2016}, it follows that the boundary of the unbounded worm domain is locally spherical. Indeed, it is easy to verify that the above map sends $(\partial \mathcal{W}_{\mathrm{un},\infty} \setminus \mathcal{A}_{\mathrm{un}})$ locally into $(\mathbb{H},(0,0)).$
\end{proof}

\section{Automorphisms of the bounded worm domain}

After studying the unbounded version of the worm domain, we can now determine the automorphisms of the Diederich--Forn\ae ss worm domain. The main result, as previously noted, is the following:

\begin{theorem} \label{autoworm}
    The automorphisms of the worm domain $\mathcal{W}$ are maps of the form
    $$F(z,w)=\left(e^{i \theta}z,w \right)$$
where $\theta \in \mathbb{R},$ i.e. rotations in the $z-$variable.
\end{theorem}

One observes that the boundary of the worm consists of the following main parts: the body, which is biholomorphic to a subset of the boundary of the unbounded worm domain, the annulus of weakly pseudoconvex points, and the two caps, together with their boundary. In particular, we introduce the following notation. Let
\begin{align*}
    \mathcal{C}_{+}=\left\{(z,w) \in \partial \mathcal{W}:  \log|z|^2 >\mu  \right\}
\text{ and }   \mathcal{C}_{-}=\left\{(z,w) \in \partial \mathcal{W}: \log|z|^2 <-\mu  \right\}
\end{align*}
denote the caps of the worm domain, whose boundaries in $\partial \mathcal{W} \setminus \mathcal{A}$ are the following tori
\begin{align*}
\mathcal{T}_{\pm}=\left\{(z,w) \in \partial \mathcal{W}: \log|z|^2 =\pm \mu, w \neq 0  \right\},
\end{align*}
and
\begin{align*}
    \mathcal{B}=\left\{(z,w) \in \partial \mathcal{W}: -\mu<\log|z|^2<\mu, w \neq 0 \right\}
\end{align*}
be the body of the worm.

We also introduce the following notation for the part of the domain lying inside the body:
\begin{align*}
    \mathcal{W}_{\mathcal{B}}=\{ (z,w) \in \mathcal{W}: -\mu<\log|z|^2<\mu\}.
\end{align*}
Since the parts of the domain lying inside the caps and the tori will not be used in the proof, we do not introduce any notation for them.

\begin{corollary} \label{locsphb}
    The main body $\mathcal{B}$ of the worm domain is locally spherical.
\end{corollary}
\begin{proof}
    This follows immediately from \Cref{locsph}.
\end{proof}

For the proof of \Cref{autoworm} we need the following lemmas:

\begin{lemma} \label{bodyconnected}
    The main body $\mathcal{B}$ of the worm domain is a connected set.
\end{lemma}
\begin{proof}
Let $z=r e^{i \theta}$ with $r>0$ and $\theta \in [0,2 \pi)$ and $e^{-\mu}< r=|z| < e^{\mu}.$
Then
\begin{align*}
    \mathcal{B} =\left\{(z,w) \in \mathbb{C}^2: \left|w-e^{i \log|z|^2} \right|^2=1, w \neq 0 \right\}= \bigcup_{e^{-\mu}<r<e^\mu} S_{r} \times [0,2\pi)
\end{align*}
where
\begin{align*}
    S_{r}=\left\{w \in \mathbb{C}: \left|w-e^{i \log r^2} \right|^2=1, w \neq 0 \right\}.
\end{align*}
Writing $w=u+iv,$ with $u,v \in \mathbb{R},$ the above set can be written as 
\begin{align*}
    S_r=\left\{(u,v) \in \mathbb{R}^2: \left( u-\cos\left(\log{r^2} \right)\right)^2+ \left( v- \sin \left(\log{r^2} \right)\right)^2=1 \right\} \setminus \{(0,0)\}.
\end{align*}

It is easy to see that $S:=\bigcup_{e^{-\mu}<r < e^{\mu}}S_{r}$ is connected, and hence $\mathcal{B}$ is connected as the product of a connected set with an interval.
\end{proof}

\begin{lemma} \label{autbody}
    The automorphisms of the main body $\mathcal{B}$ of the worm are precisely the maps of the form $(z,w) \mapsto \left(e^{i \beta}z, w \right),$ for $\beta \in \mathbb{R},$ that is, rotations in the $z$-variable.
\end{lemma}

\begin{proof}
Let $H'=(f',g') \in \mathrm{Aut}(\mathcal{B}).$ That means 
$$\left|g'(z,w)- e^{i \log|f'(z,w)|^2}\right|^2=1 \text{ for all } (z,w) \in \mathcal{B} \subset \partial \mathcal{W}_{\mathrm{un}, \infty} \setminus \mathcal{A}_{\mathrm{un}}.$$
Since $(z,w)=(1,2) \in \mathcal{B},$ by \Cref{autisobun} and \Cref{autnonisobun}, we know all local automorphisms in a neighbourhood of $(1,2).$ By the proof of \Cref{autbun}, the only candidates for global extension are the maps $H'(z,w)=\left(e^{\alpha+i \beta}z, e^{2i \alpha}w \right)$ or $H'(z,w)=\left( \frac{2^ie^{\alpha+i \beta}}{w^i z}, e^{2i \alpha}w \right)$ for $\alpha, \beta \in \mathbb{R}.$ It remains to check wether $H'(\mathcal{B}) \subseteq \mathcal{B}.$ 

We begin with $H'(z,w)=\left(e^{\alpha+i \beta}z, e^{2i \alpha}w \right).$ Clearly, $H'(\mathcal{B}) \subseteq \mathcal{B}$ when $\alpha=0.$ Assume that $\alpha>0.$ Choose $(z,w)\in \mathcal{B}$ such that $|z|$ lies in a neighborhood of $e^{\frac{\mu}{2}}$ with $\log|z|^2< \mu$ and $\log|z|^2+2\alpha>\mu.$
Then $(z,w) \in \mathcal{B},$ but $H'(z,w)=(e^{\alpha+ i \beta}z, e^{2i \alpha}w) \notin \mathcal{B},$ since $\log|e^{\alpha+i \beta}z|^2=\log|z|^2+2\alpha >\mu,$ a contradiction.
Similarly, if $\alpha<0,$ consider $|z|$ in a neighborhood of $e^{-\frac{\mu}{2}},$ which gives a contradiction, as above.
Hence, we must have $\alpha=0.$

Consider now $H'(z,w)=\left( \frac{2^ie^{\alpha+i \beta}}{w^i z}, e^{2i \alpha}w \right)$ and the family of points $(z,w)=\left(e^{\frac{t}{2}}, e^{it}(1+ e^{2 i \phi}) \right) \in \mathcal{B},$ where $t \in (-\mu,\mu)$ and $\phi \in \left(-\frac{\pi}{2}, \frac{\pi}{2} \right).$ Then $\log \left|\frac{2^ie^{\alpha+i \beta}}{w^iz} \right|^2=t+ 2(\alpha+ \phi).$ For any $\alpha \in \mathbb{R},$ choose $\pi \in \left(-\frac{\pi}{2}, \frac{\pi}{2} \right)$ such that $\phi \neq -\alpha.$ If $\alpha+ \phi>0,$ choose $t\in (-\mu,\mu)$ sufficiently close to $\mu$ so that $t+2(\alpha+ \phi)>\mu.$ Then $H'(z,w)  \notin \mathcal{B},$ since $\log \left|e^{\alpha+i \beta} \frac{2^i}{w^iz} \right|^2> \mu,$ a contradiction. Similarly, if $\alpha+\phi<0,$ choose $t \in (-\mu,\mu)$ sufficiently close to $-\mu$ so that $t+2 (\alpha+\phi)<-\mu.$ Then $H'(z,w)  \notin \mathcal{B},$ since $\log \left|e^{\alpha+i \beta} \frac{2^i}{w^iz} \right|^2<-\mu,$ again yielding a contradiction. 
Hence, the second family of maps does not satisfy $H'( \mathcal{B}) \subseteq \mathcal{B}$ for any choice of $\alpha,$ and the result follows.

\end{proof}

\begin{lemma} \label{nonsph}
   The boundary $\partial \mathcal{W}$ of the worm is not locally spherical at points $z \in \mathcal{T}_{\pm}.$
\end{lemma}
\begin{proof}
Let $p \in \mathcal{T}_+$ and suppose that there exists a locally spherical neighbourhood
$U \subseteq \partial \mathcal{W} \setminus \mathcal{A}$ of $p$. Let
$\Phi: U \to \mathbb{H}^3$ be a CR diffeomorphism into the Heisenberg hypersurface.
Then $\Phi|_{U \cap \mathcal{B}}$ is, up to local automorphisms of $\mathcal{B}$ and
$\mathbb{H}^3$, of the form
\[
(z,w) \mapsto \left(i-iz^i,\, i-2iz^i+\frac{2iz^{2i}}{w}\right),
\]
by \Cref{locsph} and \Cref{locsphb}. Strong pseudoconvexity of $\partial \mathcal{W} \setminus \mathcal{A}$ implies, by Tumanov's theorem \cite{AETumanov_1989}, that $\Phi$ admits a one-sided holomorphic extension near
$p$ which agrees with the above holomorphic expression on the side adjacent to $U \cap\mathcal{B}.$ The identity theorem then implies that this equality holds throughout a sufficiently small connected one-sided neighborhood of $p.$ By continuity up to the boundary, $\Phi|_{U\cap\mathcal{C}_+}$ is therefore given by the same expression. 
However, $\Phi$ cannot map
$U \cap \mathcal{C}_+$ into $\mathbb{H}^3$, since
$\eta(\log|z|^2)>0$ for every $(z,w) \in \mathcal{C}_+$.

The proof for $\mathcal{T}_-$ is analogous.
\end{proof}

\begin{lemma} \label{autform}
 The automorphisms of $\mathcal{W}$ are of the form $F(z,w)=\left(z A_1(w), B_1(w) \right),$ or $F(z,w)=\left(\frac{A_2(w)}{z},B_2(w) \right)$ where $A_j$ and $B_j$, $j=1,2$ are holomorphic functions.
\end{lemma}

\begin{proof}
Let $F$ be an automorphism of $\mathcal{W}.$ For $\theta \in \mathbb{R},$ let  $R_{\theta}(z,w)=\left(e^{i \theta}z,w \right)$ and define $G_{\theta}=F \circ R_{\theta} \circ F^{-1}.$ Each $G_{\theta}$ is an automorphism of $\mathcal{W}$ and hence extends smoothly to $\partial \mathcal{W} \setminus \mathcal{A},$ by Bell's Theorem \cite{bell1984local}. We continue to denote the extension by $G_{\theta}.$

Fix $p\in \mathcal{B}.$ Since $G_0$ is the identity, continuity of the map $G_{\theta}$ in $\theta$ implies that there exists an $\epsilon_p>0,$ such that $G_{\theta}(z,w) \in \mathcal{B}$ for $|\theta|<\epsilon_p.$ Continuity of $G_{\theta}$ in $(z,w)$ implies that $G_\theta(U_p) \subset \mathcal{B}$ for a small neighbourhood $U_p$ of $p$ in $\mathcal{B}.$ By \Cref{bodyconnected} and \Cref{locsphb}, $\mathcal{B}$ is connected and locally spherical and hence the image $G_{\theta}(\mathcal{B})$ is also connected and locally spherical. Moreover, $G_{\theta}(\mathcal{B}) \cap \mathcal{T}_{+}= \emptyset$ and $G_{\theta}(\mathcal{B}) \cap \mathcal{T}_{-}= \emptyset,$ since $\mathcal{T}_{+}$ and $\mathcal{T}_{-}$ consist of non locally spherical points, by \Cref{nonsph}. Hence $G_{\theta}(\mathcal{B}) \subseteq \mathcal{B}.$ Thus, by \Cref{autbody}, for every $|\theta|< \epsilon$ there exists a $\phi(\theta) \in \mathbb{R}/ 2 \pi \mathbb{Z},$ such that $G_{\theta}=R_{\phi(\theta)}.$ Tumanov's theorem \cite{AETumanov_1989} implies that $G_{\theta}(z,w)=R_{\phi(\theta)}(z,w)$ for every $(z,w) \in \mathcal{W}.$

Now, let $\theta \in \mathbb{R}$ be arbitrary. Choose $n \in \mathbb{N},$ such that $\left| \frac{\theta}{n} \right|< \epsilon_p.$ Then $$G_{\theta}=G_{\frac{\theta}{n}} \circ \cdots \circ G_{\frac{\theta}{n}}=R_{n \phi \left( \frac{\theta}{n} \right)} \in R,$$ where $R=\{R_{\theta}: \theta \in \mathbb{R}\}.$ Hence $F \circ R \circ F^{-1} \subseteq R.$ Applying the same argument to $F^{-1} \circ R_{\theta} \circ F,$ we obtain $F^{-1} \circ R \circ F \subseteq R.$ Therefore $F \circ R \circ F^{-1}=R.$ Moreover, $\phi$ is a homeomorphism, since by definition $R_{\phi(\theta_1)+\phi(\theta_2)}=R_{\phi(\theta_1)} \circ R_{\phi(\theta_2)}=R_{\phi(\theta_1+\theta_2)}$ for every $\theta_1, \theta_2 \in \mathbb{R}.$ Hence, conjugation by $F$ yields a continuous automorphism $\phi_0: S^1 \to S^1$ satisfying $\phi_0(e^{i \theta})= e^{i \phi(\theta)}.$ Since every continuous automorphism of $S^1$ is either the identity or complex conjugation, we have $\phi_0(e^{i \theta})=e^{\pm i\theta},$ and consequently $\phi(\theta)= \pm \theta \mod 2 \pi.$

It follows that for every $\theta \in [0, 2\pi),$ either $G_{\theta}=R_{ \theta}$ or $G_{\theta}=R_{-\theta}.$ Thus, either $F \circ R_{\theta}=R_{\theta} \circ F,$ or $F \circ R_{\theta}=R_{-\theta} \circ F$ for all $\theta \in [0,2\pi).$ Let $F=(F_1,F_2)$ and expand $F_1$ and $F_2$ as Laurent series $F_j(z,w)=\sum_{k =- \infty}^{\infty}a^j_k(w)z^k,$ for $j=1,2,$ the first identity implies that $F(z,w)=\left(z A_1(w), B_1(w) \right),$ and the second implies $F(z,w)=\left(\frac{A_2(w)}{z}, B_2(w) \right),$ for some holomorphic functions $A_j, B_j.$
\end{proof}

We can now prove Theorem 1:

\begin{proof}[Proof of Theorem 1.]
Let $F=(F_1,F_2)$ be an automorphism of the bounded worm domain $\mathcal{W}$. By Bell's theorem \cite{bell1984local}, $F$ extends smoothly to the strongly pseudoconvex points of $\partial\mathcal{W}$. By a slight abuse of notation, we denote this extension again by $F$, so that $
F \in C^{\infty}(\overline{\mathcal{W}} \setminus \mathcal{A}).$

We consider the following subsets of $\mathcal{B}$:
\begin{itemize}
    \item $\mathrm{I}=\left\{p \in \mathcal{B}: F(p) \in \mathcal{B} \right\}$,
    \item $\mathrm{II}_{+}=\left\{p \in \mathcal{B}: F(p) \in \mathcal{C}_{+} \right\}$,
    \item $\mathrm{II}_{-}=\left\{p \in \mathcal{B}: F(p) \in \mathcal{C}_{-} \right\}$.
\end{itemize}

Since $\partial\mathcal{W}$ is not locally spherical at the points of $\mathcal{T}_{+}$ and $\mathcal{T}_{-}$ by \Cref{nonsph}, no point of $\mathcal{B}$ can be mapped to $\mathcal{T}_{+},$ or $\mathcal{T}_-.$ Moreover, since $\mathcal{B}$ is connected, by \Cref{bodyconnected}, $F(\mathcal{B})$ is also connected. Hence $\mathcal{B}= \mathrm{I} \,\dot{\cup}\, \mathrm{II}_+ \,\dot{\cup}\, \mathrm{II}_-.$ We begin with the following Lemma.

\begin{lemma}
The sets $\mathrm{I}$, $\mathrm{II}_{+}$, and $\mathrm{II}_{-}$ are open in $\mathcal{B}$.
\end{lemma}

\begin{proof}
The set $\mathrm{I}=F^{-1}(\mathcal{B})$ is clearly open, since $\mathcal{B}$ is open in $\partial\mathcal{W}\setminus\mathcal{A}$. Since $\mathrm{II}_{+}=F^{-1}(\mathcal{C}_{+})$, it suffices to show that $\mathcal{C}_{+}$ is open in $\partial\mathcal{W}\setminus\mathcal{A}$. Indeed,
\[
\mathcal{C}_{+}
=
(\partial\mathcal{W}\setminus\mathcal{A})
\cap
\{\log|z|^{2}>\mu\},
\]
which is open in $\partial\mathcal{W}\setminus\mathcal{A}$. Similarly, $
\mathcal{C}_{-}
=
(\partial\mathcal{W}\setminus\mathcal{A})
\cap
\{\log|z|^{2}<-\mu\},$
is open in $\partial\mathcal{W}\setminus\mathcal{A}$.
\end{proof}

Suppose first that $\mathrm{I}\neq\emptyset$. Then $F(\mathcal{B})\subseteq \mathcal{B}.$ By \Cref{autbody}, $F$ is of the form $
F(z,w)=\left(e^{i\theta}z,w\right)$ with $\theta\in\mathbb{R}.$
By Tumanov's theorem \cite{AETumanov_1989}, every automorphism of the bounded worm domain is of this form.

Suppose now that $\mathrm{I}=\emptyset$. We argue by contradiction. Since $\mathrm{I}=\emptyset$, we have
$
\mathcal{B}
=
\mathrm{II}_{+}\,\dot{\cup}\,\mathrm{II}_{-}.$
As $\mathcal{B}$ is connected and $\mathrm{II}_{+}$ and $\mathrm{II}_{-}$ are open, one of these sets must be empty. Without loss of generality, $
\mathrm{II}_{+}=\mathcal{B},$
and hence $F(\mathcal{B})\subseteq\mathcal{C}_{+}.$

By \Cref{autform}, $F(z,w)=\left(z A_1(w), B_1(w) \right),$ or $F(z,w)=\left( \frac{A_2(w)}{z},B_2(w) \right)$ for holomorphic functions $A_j, B_j.$ Let $F(z,w)=\left(z A_1(w),B_1(w) \right)$ and the contradiction for the other form follows similarly.

Let $t_0 \in (-\mu,\mu),$ and let $\{r_n\}_{n=1}^{\infty} \subset \mathbb{R}$ be a sequence such that $r_n \to 0$ as $n \to \infty.$ Set $a_n= \arccos{\left(\frac{r_n}{2} \right)}$ and $w_n=r_n e^{i \left(t_0+ a_n \right)}.$ Then the points $p_n= \left( e^{\frac{t_0}{2}},w_n \right) \in \mathcal{B},$ since $\left|w_n-e^{i t_0}\right|^2=r_n^2-2 r_n \cos (a_n)+1=1.$ The assumption $F(p_n) \in \mathcal{C}_+$ implies that $\log|F_1(p_n)|^2=t_0+\log|A_1(w_n)|^2 >\mu.$ After passing to a subsequence, which we continue to denote by $\{w_n\},$ we may assume that $A_1(w_n) \to A$ and $B_1(w_n) \to B,$ for some $A, B \in \mathbb{C}.$ It follows that $t_0+c \geq \mu,$ where $c:=\log|A|^2.$

Choose $0<\delta< \min\{\mu-t_0, \frac{\pi}{2}\},$ and define $q_n(t)= \left(e^{\frac{t}{2}},w_n \right)$ for $t_0<t<t_0+\delta.$ Since $-\frac{\pi}{2}<-\delta< t_0-t<0,$ we have, for all sufficiently large $n,$ $\left|w_n - e^{it} \right|^2=r_n^2-2r_n \cos (t_0+a_n-t)+1 <1.$ Hence $q_n(t) \in \mathcal{W}_{\mathcal{B}}$ for all sufficiently large $n.$ For every fixed $t \in (t_0,t_0+ \delta),$ we have $F(q_n(t))= \left(e^{\frac{t}{2}}A_1(w_n), B_1(w_n) \right).$ Since $q_n(t)$ converges to a boundary point of $\mathcal{W},$ it eventually leaves every compact subset of $\mathcal{W}.$ By properness of the biholomorphism $F: \mathcal{W} \to \mathcal{W},$ the sequence $F\left(q_n(t)\right)$ also leaves every compact subset of $\mathcal{W}.$ Since it converges in $\mathbb{C}^2,$ the limit must therefore belong to $\partial \mathcal{W}.$

Using the defining equation of $\partial \mathcal{W},$ we obtain $\left|B-e^{i(t+c)} \right|^2=1- \eta(t+c).$ Equivalently, $\eta(t+c)=2 \operatorname{Re}\left(B e^{-i(t+c)} \right)-|B|^2.$ Differentiating twice  with respect to $t,$ we obtain $\eta''(t+c)=-\eta(t+c)-|B|^2.$ Since $\eta$ is non-negative and convex, the right-hand side must vanish. In particular, $B=0.$ It follows that $\eta(t+c)=0,$ and hence $t+c \leq \mu.$ 

Combining this inequality with the one obtained by the sequence of boundary points, we conclude that $\mu \leq t_0+c< t+c \leq \mu,$ which is a contradiction. Therefore, $\mathrm{I} \neq \emptyset,$ and the result follows.
    
\end{proof}

It is evident that the maps $(z,w) \mapsto \left(e^{i \theta} z, w \right)$ extend as $C^{\infty}$ diffeomorphisms to the entire boundary, even though the boundary of the worm contains a weakly pseudoconvex annulus of points and does not satisfy Condition $\mathrm{R}$.

\section*{AI Declaration}
OpenAI's ChatGPT (GPT-5.6 and GPT-6) was used to assist with language editing, grammatical corrections, and improving the clarity of the paper. All mathematical results, proofs, and arguments were developed and verified independently by the author.

\bibliographystyle{amsplain}
\bibliography{references}
\end{document}